\documentclass[11pt]{article}
\usepackage{mathrsfs}

\usepackage{amsmath,amssymb,graphicx,epsfig,cite,rotating,setspace,amsthm,color}

\usepackage{amsmath,amssymb,amsthm,mathrsfs}
\usepackage{hyperref,cite}
\usepackage{geometry}
\usepackage{color,graphicx}

\newcommand{\beq}{\begin{equation}}
\newcommand{\eeq}{\end{equation}}
\newcommand{\bea}{\begin{eqnarray}}
\newcommand{\eea}{\end{eqnarray}}

\newcommand{\RM}{{\mathbb{R}}}

\newcommand{\HT}{H^1(\mathbb{R})}

\newcommand{\eq}[2]{\begin{equation}\begin{split}#1\end{split}\label{#2}\end{equation}}
\newcommand{\eqnn}[1]{\begin{equation}\begin{split}#1\end{split}\nonumber\end{equation}}

\numberwithin{equation}{section}

\newtheorem{theorem}{Theorem}[section]

\newtheorem{lemma}[theorem]{Lemma}

\theoremstyle{definition}

\theoremstyle{remark}

\usepackage{authblk}
\title{Orbital Stability of Smooth Traveling Solitary Waves to the\ Fornberg--Whitham Equation}
\author[1]{Xijun Deng\thanks{Email: \texttt{xijundeng@yeah.net}}}
\author[2]{St{{\'e}}phane Lafortune\thanks{Email: \texttt{lafortunes@cofc.edu}}}
\author[3]{Zhisu Liu\thanks{Email: \texttt{liuzhisu@cug.edu.cn}}}
\affil[1]{Department of Mathematics, Hubei University of Automotive Technology, Shiyan, Hubei 442002, P. R. China}
\affil[2]{Department of Mathematics, College of Charleston, Charleston, SC 29401, USA}
\affil[3]{Center for Mathematical Sciences, School of Mathematics and Physics, China University of Geosciences, Wuhan, Hubei 430074, P. R. China}
\date{}

\begin{document}
\maketitle

\begin{abstract}
The Fornberg--Whitham (FW) equation
was introduced by Fornberg and Whitham~\cite{FornbergWhitham1978} as a nonlocal model for unidirectional shallow water waves capable of capturing wave steepening and breaking. Despite its similarities with integrable shallow-water equations, the FW equation is not completely integrable. Nevertheless, the FW equation is part of the family of peakon-type models as it supports peaked traveling wave solutions. In this paper, we consider smooth solitary wave solutions to the FW equation. We use a variational approach to show that {{some}} are orbitally stable.
\end{abstract}

\section{Introduction}\label{sec:intro}
The Fornberg--Whitham (FW) equation
\begin{equation}\label{eq:FW}
 u_t - u_{txx} + u_x + u u_x = u u_{xxx} + 3 u_x u_{xx}
\end{equation}
was introduced by Fornberg and Whitham~\cite{FornbergWhitham1978} as a nonlocal model for unidirectional shallow water waves capable of capturing wave steepening and breaking. In contrast with classical weakly dispersive models such as KdV, the FW equation incorporates nonlocal dispersion through the Helmholtz operator $(1-\partial_x^2)^{-1}$, placing it within the broader class of Whitham-type equations derived to better approximate the full water-wave dispersion relation.

Applying $(1-\partial_x^2)^{-1}$ to \eqref{eq:FW} yields the equivalent nonlocal formulation
\begin{equation}\label{eq:FW-nonlocal}
 u_t + u u_x +(1-\partial_x^2)^{-1} u_x=0,
\end{equation}
where the convolution kernel is given explicitly by $\tfrac12 e^{-|x|}$. This representation highlights the balance between nonlinear steepening and nonlocal dispersion, and it has been used for the analysis of well-posedness, wave breaking, and weak solution concepts for the FW equation; see, for instance, Holmes and Thompson, H\"ormann, and Yang~\cite{HolmesThompson2014,Hormann2020,Yang2020}.

Despite its similarities with integrable shallow-water equations, the FW equation is not completely integrable. In particular, Ivanov~\cite{Ivanov2005} showed that within a broad class of nonlinear dispersive equations of Camassa--Holm type, the only completely integrable cases correspond to the Camassa--Holm and Degasperis--Procesi equations. Consequently, the FW equation does not admit a Lax pair or a bi-Hamiltonian structure, and it lacks the infinite hierarchy of conservation laws characteristic of integrable peakon equations.

Nevertheless, the FW equation is part of the family of peakon-type models. Like the Camassa--Holm and Degasperis--Procesi equations~\cite{CamassaHolm1993,DegasperisProcesi1999}, as well as the Novikov and modified Camassa--Holm equations~\cite{Fuchssteiner1996,Novikov2009}, the FW equation supports peaked traveling wave solutions. In fact, Fornberg and Whitham already identified a peaked solitary wave as a limiting profile of smooth traveling waves~\cite{FornbergWhitham1978}. Subsequent work has produced explicit peakons, cuspons, and composite waves, underscoring that FW shares many qualitative features with integrable peakon equations, even though it lacks their algebraic integrability.

The Cauchy problem for the Fornberg--Whitham (FW) equation has been studied extensively over the past decade, beginning with local well-posedness results in Sobolev spaces. In the periodic setting, Holmes~\cite{Holmes2016} established local existence, uniqueness, and continuous dependence on the initial data in $H^s$ for $s>\tfrac{3}{2}$. These results were subsequently extended by Holmes and Thompson~\cite{HolmesThompson2016}, who proved local well-posedness in Besov spaces {$B^s_{2,r}$ for $s>\tfrac{3}{2}, r\in(1,\infty)$ or $s\geq\frac{3}{2}, r=1$}, both on the torus and on the real line.
{{It is important to note that when $r = 2$, this corresponds to the Sobolev space $H^s$ where well-posedness has been shown by Yin \cite{Yin2005} by applying Kato-semigroup approach.}} In particular \cite[Theorem 2.2]{Yin2005}, for initial data $u_0\in H^s(\mathbb{R})$, there exists $T>0$ such that \[ u\in C([0,T];H^s(\mathbb{R})), \] and hence, by Sobolev embedding, \[ u,u_x\in C([0,T];L^\infty(\mathbb{R})). \]
More recently, Guo~\cite{Guo2023} investigated the FW equation in general Besov spaces and obtained local well-posedness in
\[
B^s_{p,r}(\mathbb{R}), \qquad s>1+\tfrac{1}{p}, \quad 1\le p\le\infty,\quad 1\le r<\infty.
\]
Furthermore, Qu, Wu, and Xiao~\cite{QuWuXiao2024} proved local well-posedness in the Besov space $B^1_{\infty,1}(\mathbb{R})$. A comprehensive synthesis of several of these developments---including a comparison of strong, mild, and weak solution concepts and their relation to wave-breaking mechanisms---is provided in the survey article by H\"ormann~\cite{Hormann2021}.

This behavior stands contrast to that of other peakon equations such as the Camassa--Holm (CH) equation and, more generally, the $b$-family \eqref{bf}. These equations are known to be locally well-posed in $H^s(\mathbb{R})$ for $s>\tfrac{3}{2}$~\cite{30,31,32,33}, and ill-posed in $H^s(\mathbb{R})$ for $s<\tfrac{3}{2}$~\cite{Him} (and for $1<b\le 3$, ill-posedness persists when $s=3/2$~\cite{GuoLiuMolinetYin2019}). In the specific case of the Camassa--Holm equation, local well-posedness holds in the Besov spaces $B^{1+\frac{1}{p}}_{p,1}(\mathbb{R})$ for $1\le p<\infty$ \cite{YeYinGuo2023}, while the equation is ill-posed in  $B^1_{\infty,1}(\mathbb{R})$~\cite{Guo2022,LiYuZhu}. In addition, norm inflation occurs in the critical Sobolev and Besov spaces
$
B^{1+\frac{1}{p}}_{p,r}(\mathbb{R}),  1\le p\le\infty, 1<r\le\infty,
$
as shown in~\cite{GuoLiuMolinetYin2019}.

Finally, we note that, unlike other shallow water models with peakon solutions such as the Camassa--Holm and Degasperis--Procesi equations and, more generally, the $b$-family, no general global existence result for smooth solutions of the Fornberg--Whitham equation is currently available. Furthermore, the issue of wave breaking for the FW equation well-studied (see for example \cite{1,2}).

A systematic classification of traveling wave solutions to \eqref{eq:FW} was carried out by Yin, Tian, and Fan~\cite{YinTianFan2010}, who identified parameter regimes supporting smooth periodic waves, smooth solitary waves, as well as various nonsmooth profiles including peaked and cusped waves. Their analysis shows that smooth solitary waves arise as limits of smooth periodic families and are separated from nonsmooth solutions by the singular set $\{u=c\}$, where the traveling-wave reduction degenerates.

Beyond existence and classification, stability properties of solitary waves have also been investigated.  {Variational approaches have been successfully employed to establish the orbital stability of smooth solitary waves for related equations, such as the Novikov equation ~\cite{Lafortune2024} and the Degasperis-Procesi equation ~\cite{CharalampidisParkerKevrekidisLafortune2023}. For the FW equation,} using variational methods combined with a penalization argument and concentration--compactness, Zhang, Xu, and Li~\cite{ZhangXuLi2021} construct solitary waves for speeds $c>1$ as minimizers of a Lyapunov functional under a fixed $L^2$ constraint, and prove orbital stability of the resulting minimizing set. Their approach begins by postulating an appropriate energy functional, showing that the associated constrained minimization problem admits minimizers, and then establishing stability for all elements of this minimizing set. In contrast, in the present work we restrict attention to smooth traveling wave solutions, provide a complete characterization of such waves as critical points of a variational functional, and apply the Grillakis--Shatah--Strauss framework to show that all smooth solitary waves in this class are orbitally stable.

The purpose of the present work is to develop an orbital stability theory for \emph{smooth} traveling solitary waves of the Fornberg--Whitham equation. We obtain an explicit traveling-wave reduction together with a variational characterization of the smooth profiles, placing the FW equation within the framework of stability theory for nonintegrable peakon-type models.

The paper is organized as follows. In Section~\ref{sec:tw}, we obtain the smooth traveling solitary waves. In Section~\ref{sec:var}, we establish their variational characterization, and in Section~\ref{sec:stab}, we perform the spectral analysis. In Section~\ref{VK}, we verify the slope condition, also known as the Vakhitov--Kolokolov condition, which guarantees positivity of the Hessian on the appropriate constraint space. In Section~\ref{sec:L2-stability}, we establish $L^2$--$L^\infty$ orbital stability. Finally, Section~\ref{conc} is devoted to conclusions and discussion.

\section{Traveling-wave reduction and smooth solitary waves}\label{sec:tw}

In this section, we first construct traveling wave solutions on a nonzero background. Although a symmetry of the FW equation allows any such solution to be transformed into one with zero background, working initially in the nonzero background setting serves two important purposes. First, it makes explicit the full family of traveling waves to which our stability criterion applies, without prematurely quotienting by the symmetry. Second, it reveals that a subset of these solutions naturally arise on negative backgrounds, a feature that would be obscured if one restricted attention to the zero background case from the outset.

\subsection{Traveling-wave solutions with nonzero background}

We seek traveling-wave solutions of \eqref{eq:FW} of the form
\begin{equation}\label{eq:tw-ansatz}
 u(x,t) = \phi(\xi), \qquad \xi = x-ct,
\end{equation}
where $c\in\mathbb{R}$ is the wave speed and $\phi$ is the profile. Substituting \eqref{eq:tw-ansatz} into \eqref{eq:FW} yields a third-order ODE for $\phi$ (see \cite[Eq.~(2.1)]{YinTianFan2010}):
\begin{equation}\label{eq:YTF-2.1}
 -c\,\phi' + c\,\phi''' + \phi' + \phi\,\phi' = \phi\,\phi''' + 3\,\phi'\,\phi''.
\end{equation}
Integrating \eqref{eq:YTF-2.1} once with respect to $\xi$ gives the second-order relation (cf. \cite[Eq.~(2.2)]{YinTianFan2010})
\begin{equation}\label{eq:YTF-2.2}
\left( \big(\phi-c\big)^2\right)'' = (\phi  -c)^2+2\phi +\alpha,
\end{equation}
{{where $\alpha\in\mathbb{R}$ is an integration constant.}}

Following \cite{YinTianFan2010}, we multiply \eqref{eq:YTF-2.2} by $\big((\phi-c)^2\big)'$ and integrate once more. This yields the first-order energy relation
\eq{
2 \phi'^2 &+ F(\phi)=\alpha,
\\
 F(\phi) &= -\frac{(\phi-c)^2}{2} +\left( \frac{2\phi^2(3c-2\phi)/3+\beta}{(\phi-c)^2}\right),
}{eq:YTF-2.3}
for some additional constant $\beta\in\mathbb{R}$. The first-order invariant (\ref{eq:YTF-2.3})
represents the energy conservation for a Newtonian particle
with the mass $2$ and energy $\alpha$ under a force with
the potential energy $F$.

Smooth solitary wave solutions with profile $\phi \in C^{\infty}(\mathbb{R})$ satisfying $\phi(x) \to k$ as $|x| \to \infty$ correspond to the homoclinic
orbit from the equilibrium point $(\phi,\phi') = (k,0)$. Taking the limit as $|x|\rightarrow \infty$ in \eqref{eq:YTF-2.2} and \eqref{eq:YTF-2.3}
yields the relations:
\begin{equation}
\label{ag}
	\alpha = -(c-k)^2-2k, \qquad \beta=-(c-k)^4/2 -k(2c-k)^2/2-k^3/6.
\end{equation}

In order to ensure that the system \eqref{eq:YTF-2.3} has such a homoclinic orbit, one needs to verify that the ``potential function'' $F$ has a local maximum and a local minimum.  Critical points of $F$ in $(-\infty,c)$ are given by roots of the algebraic equation
\begin{equation}
-3(c-\phi)^4-3\phi(2c-\phi)^2-\phi^3 = 6\beta
\label{algebraic}
\end{equation}
 for $\phi \in (-\infty,c)$. { The mapping $M(\phi) : (-\infty,c) \mapsto \mathbb{R}$ defined by
$M(\phi) \equiv -3(c-\phi)^4-3\phi(2c-\phi)^2-\phi^3 $ has one critical point given by a global maximum  that
occurs at $\phi_c \equiv  c-1<c$. Furthermore, $M(\phi)\rightarrow -\infty$ as $\phi\rightarrow -\infty$ and $M(c)=-4c^3$. Therefore, there exists only one critical point of $F$ for $\beta \in (-\infty,-2c^3/3)$ and $\beta = \mathfrak{b}$, two critical points of $F$  for $\beta \in (-2c^3/3,\mathfrak{b})$, and no critical points
of $F$  for $\beta \in (\mathfrak{b},\infty)$, where
\begin{equation*}
\mathfrak{b} \equiv  \frac{M(\phi_c)}{6}= \frac{1}{6}-\frac{2c^3}{3}.
\end{equation*}}
Homoclinic orbits with $\phi \in (-\infty,c)$ exist only if at least two roots of the algebraic equation \eqref{algebraic} exist in $(-\infty,c)$, which
happens if and only if $\beta \in (-2c^3/3,\mathfrak{b})$. The two roots
can be ordered as follows:
\begin{equation*}
\phi_1=k < c-1 < \phi_2 < c.
\end{equation*}
Implementing that $\beta>-2c^3/3$ using the expression \eqref{ag}, one finds $k>c-4/3$, and $\beta<\mathfrak{b}=1/6-2c^3/3$ using the expression \eqref{ag}, one finds $k<c-1$. Thus the free parameter $k$ satisfies the condition
\begin{equation}
\label{kin}
k\in\left(c-\frac43,\,c-1\right).
\end{equation}
 For fixed $c > 0$, the local maximum and minimum points of $F$ gives respectively the saddle point $(\phi_1,0)$ and the center point $(\phi_2,0)$ of the second-order equation (\ref{eq:YTF-2.2}). Thus, $\phi_1 \equiv k \in (c-4/3,c-1)$ is taken as the arbitrary parameter of the homoclinic orbit, satisfying \eqref{kin},
which specifies $\alpha$ and $\beta$ by the relation (\ref{ag}). Since $F(\phi) \to +\infty$ as $\phi \to c$ from the left, the homoclinic orbit belongs to the vertical stripe $\{ (\phi,\phi') : c-4/3 < \phi < c \}$ and represents the smooth solitary wave with the profile $\phi \in C^{\infty}(\mathbb{R})$ satisfying  $\phi(x) \to k$ as $|x| \to \infty$. By the translational invariance, the solitary wave profile satisfying $\phi'(0) = 0$ is uniquely defined. Figure \eqref{fig1} shows the graph of $F$ in the case $c=2$ and $k=5/6$.

The maximum value of $\phi$ is found by letting $\phi=$constant in  \eqref{eq:YTF-2.3}, thus solving
 \eqnn{\alpha-F(\phi)=0\implies (\phi-k)^2\left(-3 \phi^{2}+\left(12 c -6 k -8\right) \phi -12 c^{2}+12 k c -3 k^{2}+12 c -4 k\right)=0,}
where $F$ is given in \eqref{eq:YTF-2.3}. The question above has a double root at $\phi=k$ and, under the condition \eqref{kin}, it has a real root below and a root above $c$. We denote the root below $c$ by $\phi_{max}$ and it is given by
\eq{
\phi_{kmax}\equiv2 c -k -\frac{4}{3}-\frac{2 \sqrt{3(k-c+4/3)}}{3}.
}{phite2}

We have the following existence lemma.

\begin{figure}[h]
\begin{center}
	\includegraphics[width=0.6\textwidth]{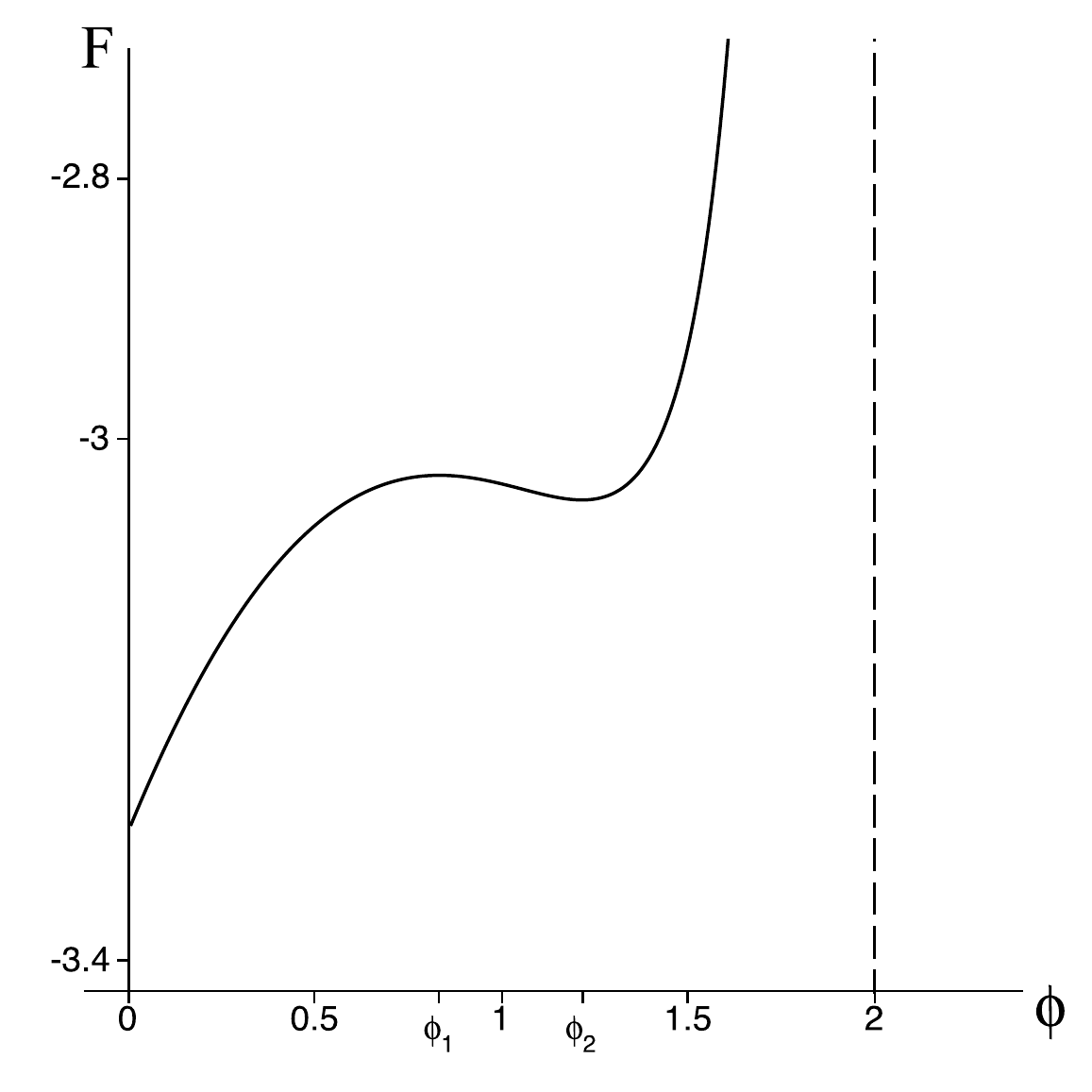}
	\end{center}\caption{Left: $F$ versus $\phi$ for $c=2$ and $k=5/6$, with $\beta=-5.2$ as given by \eqref{ag}.  The two critical points are labeled $\phi=\phi_1=k$ and $\phi=\phi_2$. \label{fig1}}
\end{figure}

\begin{lemma}
	\label{lem-travi}
	For fixed  $c > 0$, there exists a one-parameter
	family of smooth solitary waves with profile $\phi_k \in C^{\infty}(\mathbb{R})$ satisfying $\phi_k'(0) = 0$ and
	$\phi_k(x) \to k$ as $|x| \to \infty$ if and only if the arbitrary parameter $k$ belongs
	to the interval $(c-4/3,\,c-1)$. Moreover,
	{\begin{equation*}
	k <	\phi(x) < c,
	\end{equation*}}
	and the family is smooth with respect to parameter $k$. Moreover,  $\sup_{x\in \mathbb{R}}\left(\phi_k\right)=\phi_{kmax}$, where $\phi_{kmax}$ is given by \eqref{phite2}.
\end{lemma}


 \subsection{Symmetry and traveling-wave solutions with zero background}
 \label{sec:2.2}

The FW equation \eqref{eq:FW} or, the equivalent form \eqref{eq:FW-nonlocal}, admit the following symmetry
\eq{
\text{if }u(x,t)\text{ is a solution then so is }u(x-kt,t)+k.
}{sym}
The symmetry above can be applied to the solutions of Lemma \ref{lem-travi} to obtain the following solutions with zero background, which we denote by $\phi_c$
{{\begin{lemma}
\label{lem-trav}
For fixed $c\in(1,4/3)$, there exists a one-parameter family of smooth solitary waves
$
    \phi_c\in C^{\infty}(\mathbb{R}),
$
satisfying
$\phi_c'(0)=0,
    \phi_c(x)\to 0$ as $|x|\to\infty$.
Moreover, the famility is smooth with respect to $c$, and
$
    \sup_{x\in\mathbb{R}}\phi_c(x)=\phi_{c\mathrm{max}},
$
where $\phi_{c\mathrm{max}}$ is given by \eqref{phite} below. In addition, the map
\[
     c\longmapsto \phi_c
    \in L^{\infty}(\mathbb{R})\cap L^2(\mathbb{R})
\]
is smooth.
\end{lemma}

\begin{proof}
The existence and the properties of the profiles, including the Pointwise smoothness in $c$,  follow from
Lemma~\ref{lem-travi} and the application of the symmetry \eqref{sym}. We only verify the
smooth dependence of $c\mapsto\phi_c$ as a map into
$L^\infty(\mathbb{R})\cap L^2(\mathbb{R})$.

Consider a compact interval
containing $I\subset (1,4/3)$. It follows from the first order profile equation \eqref{eq:YTF-2.3} that
$\phi_c\sim e^{\pm x\sqrt{(c-1)/c}}$ as $x\to\mp \infty$.
Thus, for every $j\geq0$,  there exist
constants $C_j,\alpha_j>0$ such that
\[
    \left|\partial_c^j\phi_c(x)\right|
    \leq C_j e^{-\alpha_j|x|},
    \qquad c\in I,\quad x\in\mathbb{R}.
\]

We show that the pointwise derivatives are the derivatives in
$L^\infty(\mathbb{R})\cap L^2(\mathbb{R})$. Taylor's
theorem gives
\[
\left|
\frac{\partial_c^j\phi_{c+h}(x)-\partial_c^j\phi_c(x)}{h}
-\partial_c^{j+1}\phi_c(x)
\right|
\leq
\frac{|h|}{2}\sup_{c\in I}
\left|\partial_c^{j+2}\phi_{c}(x)\right|\leq C|h|e^{-\alpha|x|}.
\]
for come $C,\alpha>0$, where we used the exponential bound above.
Hence,
\[
\left\|
\frac{\partial_c^j\phi_{c+h}-\partial_c^j\phi_c}{h}
-\partial_c^{j+1}\phi_c
\right\|_{L^\infty\cap L^2}
\longrightarrow0
\qquad\text{as }h\to0.
\]
\end{proof}

}}

The maximum value of the solution for the solution with zero background is obtained by setting $k$ to zero in \eqref{phite2}
\eq{
\phi_{\text{max}}\equiv2 c -\frac{4}{3}-\frac{2 \sqrt{4-3c}}{3}.
}{phite}

{{Note that when $c$ approaches $4/3$, the smooth solitary wave approaches the peakon solution \cite{FornbergWhitham1978,Hormann2018,ChenLiHuang2012,ZhouTian2010}. }}

In the rest of this paper, we show that some of the solution of Lemma \ref{lem-trav} is orbitally stable, which will imply the stability of the more general solution of Lemma \ref{lem-travi} since they are related by a point transformation.

\section{Hamiltonian structure}\label{sec:var}
\subsection{Hamiltonian formulation}\label{HF1}
It is convenient to write a Hamiltonian formulation that is compatible with the nonlocal representation \eqref{eq:FW-nonlocal}.
We observe that one may write the FW dynamics  in the Hamiltonian form on $\HT$
 as
\begin{equation}\label{HF}
 u_t = \mathcal{J}\,E'(u), \qquad \mathcal{J}=\partial_x,
\end{equation}
where the energy functional is
\begin{equation}
\label{Energy}
 E(u)\equiv\int_{\mathbb{R}}\Big(-\tfrac16 u^3-\tfrac12\,u\,(1-\partial_x^2)^{-1}u\Big)\,dx,
\end{equation}
with {the inverse} of the Helmotz operator  $1-\partial_x^2$ being given by
$$
(1-\partial_x^2)^{-1}v=G(x)\ast v,\;\;G(x)\equiv \frac{e^{-|x|}}{2}.
$$
Moreover, we have the mass-type quantity
\begin{equation}
\label{charge}
 Q(u)=\frac12\int_{\mathbb{R}}u^2\,dx
\end{equation}
is conserved along sufficiently regular solutions.


Using the Fourier transform, we compute
\[
\int_{\mathbb{R}} u\,(1-\partial_x^2)^{-1}u\,dx
=
\int_{\mathbb{R}} \frac{1}{1+|\xi|^2}\,|\hat{u}(\xi)|^2\,d\xi
\le
\int_{\mathbb{R}} |u(x)|^2\,dx,
\]
which shows that this term in the energy functional $E$ is well defined on $\HT$.
Furthermore, the functionals defined in \eqref{Energy} and \eqref{charge} are
 twice Fr\'echet differentiable on $\HT$, with first and second
variations given in \eqref{crit} and \eqref{sec}. This regularity follows from
the continuous embedding $\HT \hookrightarrow L^\infty(\mathbb{R})$, see
\cite[Theorem~8.8]{Brezis2011}, which implies that the cubic term appearing in
$E$ is bounded by a constant multiple of $\|u\|_{\HT}^3$.


\subsection{Variational characterization via constrained minimization}\label{Var}

 Consider the following combination of the above conserved quantities
\eq{
{\mathcal{H}}(u)\equiv {E}(u)+c\,{Q}(u).
}{L}
The traveling wave solution $u=\phi_c(x-ct)$ specified by Lemma \ref{lem-trav} is a critical point of  ${\mathcal{H}}$
{given above in the sense that for any $\varphi\in \HT${
\eq{
\left<{\mathcal{H}}'(\phi_c),\,\varphi\right>=\int_{\mathbb{R}}\big(-\frac{\phi_c^2}{2}-(1-\partial_x^2)^{-1}\phi_c+c\phi_c\big)\varphi\,dx=0.
}{crit}}}
It is indeed straightforward to verify that relation \eqref{crit} is equivalent to the traveling wave equation \eqref{eq:YTF-2.3} satisfied by $\phi_c$ in the case $k=0$.
Furthermore, the quantity $Q$ can be obtained from the general theory
of Grillakis, Shatah, and Strauss \cite{GSS1987} as the conserved functional associated to the space translation Hamiltonian symmetry
\eq{
T(s)u(x,t)\equiv u(x+s,t).
 }{SymD}
Using \cite[Equation (2.9)]{GSS1987}, one finds
\eqnn{
Q= \frac12\left<\mathcal{J}^{-1} T'(0) u,\,u\right>=\frac12\left<u,\,u\right>.
}
As such, we are able to apply the Grillakis-Shatah-Strauss theory directly. We briefly review this theory below.

\section{Spectral Analysis}\label{sec:stab}

In this section, we analysis the spectrum of the linear operator arising from the second variation of the Lyapunov functional
introduced in Section \ref{Var}
\eq{
{\mathcal{L}}\equiv \left.{\mathcal{H}}''\right|_{u=\phi_c}=(c-\phi_c)-(1-\partial_x^2)^{-1}.
}{sec}
{A related spectral analysis of the Hessian associated with smooth solitary waves of the Degasperis-Procesi equation can be found in~\cite{Li2020DP}. In lemmas \ref{ess} and \ref{point}, we establish that the essential spectrum of $\mathcal L$ is positive and bounded away from zero, that its kernel is spanned by $\phi_c'$, and that it has exactly one simple negative eigenvalue.}

\subsection{Essential Spectrum}
\label{essS}
The essential spectrum has two contributions: the asymptotic
constant-coefficient operator and the range of the variable coefficient
$c-\phi_c(x)$.

\begin{lemma}\label{ess}
For $c\in(1,4/3)$, the essential spectrum of $\mathcal{L}$ in \eqref{sec} is
\[
 {\displaystyle
 \sigma_{\mathrm{ess}}(\mathcal{L})
 = [c-1,c]\cup[c-\phi_{\max},c]
 = [\min\{c-1,c-\phi_{\max}\},c].}
\]
In particular,
$$
\inf\sigma_{\mathrm{ess}}(\mathcal{L})
=\min\{c-1,c-\phi_{\max}\}>0.
$$
\end{lemma}

\begin{proof}
We apply the limit-operator Fredholm criterion of Lange and Rabinovich
\cite[Corollary~1.1]{LangeRabinovich1986} (for more recent statements of that result, see for example \cite[Theorem 4.3.16]{Rabinovich2004}, \cite[Theorems 3.4 and 3.5]{RabinovichRoch2008}, and \cite[Theorem 1.1]{Mantoiu2013}). In the present setting, the
corollary states that a bounded pseudodifferential operator is Fredholm
if and only if all of its limit operators, obtained by translating its
symbol to infinity in phase space, are invertible with uniformly bounded
inverses. Applied to $\mathcal L-\lambda I$, this implies that
$\lambda$ belongs to the essential spectrum of $\mathcal L$ precisely
when $\lambda$ belongs to the spectrum of at least one of these limit
operators.

The operator
\[
\mathcal L=(c-\phi_c)-(1-\partial_x^2)^{-1}
\]
is a bounded self-adjoint pseudodifferential operator on $L^2(\mathbb R)$
with symbol
\[
\ell(x,\xi)
=
c-\phi_c(x)-\frac{1}{1+\xi^2}.
\]
Since $\ell\in C_b^\infty(\mathbb{R}_x\times\mathbb{R}_\xi)$, it belongs to
the order-zero symbol class considered in \cite[Corollary~1.1]{LangeRabinovich1986}, and
the corresponding limit-operator criterion applies. Its limit operators are obtained by considering
sequences $(x_n,\xi_n)$ tending to infinity in phase space. Since the
symbol is the sum of a function of $x$ and a function of $\xi$, there
are only three types of limits: $|x_n|\to\infty$ with $\xi_n$ bounded,
$|\xi_n|\to\infty$ with $x_n$ bounded, and both components tending to
infinity.

Since $\phi_c(x)\to0$ as $|x|\to\infty$, the spatial limit is the
Fourier multiplier with symbol
\[
\ell_\infty(\xi)
=
c-\frac{1}{1+\xi^2}.
\]
Its spectrum is the closure of its range, namely
\[
\Sigma_\infty=[c-1,c].
\]

On the other hand, since
\[
\frac{1}{1+\xi^2}\longrightarrow0
\qquad\text{as }|\xi|\to\infty,
\]
the high-frequency limit is the multiplication operator with symbol
\[
\ell_{\rm hf}(x)=c-\phi_c(x).
\]
Since $0<\phi_c(x)\leq\phi_{\max}$ and $\phi_c(x)\to0$ as
$|x|\to\infty$, its spectrum is
\[
\Sigma_{\rm hf}=[c-\phi_{\max},c].
\]
The limit in which both $|x|$ and $|\xi|$ tend to infinity gives the
constant operator $cI$ and hence contributes only the common endpoint
$c$.

Let \[ \Sigma=\Sigma_\infty\cup\Sigma_{\rm hf}. \]
To apply Corollary~1.1 of \cite{LangeRabinovich1986} and conclude that $\Sigma$ is the essential spectrum, it remains only to verify
the uniform boundedness of the resolvents of the limit operators. If $\lambda\notin\Sigma$,
then every limit operator $\mathcal L_h$ is self-adjoint and satisfies
\[ \sigma(\mathcal L_h)\subset\Sigma. \]
Hence
\[ \|(\mathcal L_h-\lambda I)^{-1}\| = \frac{1}{\operatorname{dist}(\lambda,\sigma(\mathcal L_h))} \leq \frac{1}{\operatorname{dist}(\lambda,\Sigma)}, \]
uniformly in $\mathcal L_h$, where we used the spectral theorem.
Thus the hypotheses of Corollary~1.1 of \cite{LangeRabinovich1986} are satisfied, and therefore \[ \sigma_{\mathrm{ess}}(\mathcal L)=\Sigma. \]

Therefore,
\[
\sigma_{\mathrm{ess}}(\mathcal L)
=
\Sigma_\infty\cup\Sigma_{\rm hf}
=
[c-1,c]\cup[c-\phi_{\max},c]
=
[\min\{c-1,c-\phi_{\max}\},c].
\]

Finally, by \eqref{phite},
\[
c-\phi_{\max}
=
\frac43-c+\frac23\sqrt{4-3c}>0
\]
for $1<c<4/3$, while $c-1>0$. Hence
\[
\inf\sigma_{\mathrm{ess}}(\mathcal L)
=
\min\{c-1,c-\phi_{\max}\}>0.
\]

\end{proof}

 \subsection{Point Spectrum}

 We first present a result about the whole spectrum of ${\mathcal{L}}$.
\begin{lemma}
\label{Lemspec}
The spectrum of the operator
\eq{
{\mathcal{L}}=(c-\phi_c)-(1-\partial_x^2)^{-1}.
}{sec2}
acting on $L^2(\RM)$ is real and satisfies
\begin{equation}\label{Sspec}
\sigma\left(\mathcal{L}\right)\subset \left[\lambda_0,\,c\right],\;\;\;\lambda_0\equiv c-1-\phi_{\text{max}}
\end{equation}
where $\phi_{\text{max}}=\sup_{x\in \mathbb{R}}\left(\phi_c\right)$ is specified in \eqref{phite}.
\end{lemma}
Note that the interval in \eqref{Sspec} has a nonzero intersection with the negative real axis since $\lambda_0=c-1-\phi_{\text{max}}<0$.

\begin{proof} The operator $\mathcal L$ is bounded and self-adjoint on $L^2(\mathbb R)$. For every $v\in L^2(\mathbb R)$, the inequalities
$0\leq \phi_c(x)\leq\phi_{\max}$
give \[ (c-\phi_{\max})\|v\|_{L^2}^2 \leq\left<(c-\phi_c(x))v(x),\,v\right> \leq c\|v\|_{L^2}^2. \]
Moreover, by Plancherel's theorem,
\[ \left\langle(1-\partial_x^2)^{-1}v,v\right\rangle = \int_{\mathbb R} \frac{|\widehat v(\xi)|^2}{1+\xi^2}\,d\xi\implies 0\leq \left\langle(1-\partial_x^2)^{-1}v,v\right\rangle \leq \|v\|_{L^2}^2. \]
 It follows that
\[ (c-1-\phi_{\max})\|v\|_{L^2}^2 \leq \langle\mathcal Lv,v\rangle \leq c\|v\|_{L^2}^2 \]
for every $v\in L^2(\mathbb R)$. Since $\mathcal L$ is bounded and self-adjoint, the
spectral theorem yields \[ \sigma(\mathcal L) \subset [c-1-\phi_{\max},c]. \]
\end{proof}

%
%

 \begin{lemma}
 \label{point}
The spectrum of the operator $\mathcal{L}$  given in \eqref{sec} on $L^2(\mathbb{R})$ is real. Its kernel is  one-dimensional and spanned by $\phi_c'$, with $\phi_c$ specified by Lemma \ref{lem-trav}. The point spectrum contains only one negative eigenvalue.
\end{lemma}

 \begin{proof}

{
By Lemmas \ref{ess} and \ref{Lemspec}, every nonpositive spectral point lies in $[\lambda_0,0]$and is an isolated eigenvalue of finite multiplicity.

We consider
\begin{equation}\label{eq:FW-eigen}
 \mathcal{L}v=\lambda v,\qquad \lambda\in[\lambda_0,0],
 \quad v\in L^2(\mathbb R).
\end{equation}
Define
$$
 p=Kv,\qquad
 A(x,\lambda)
 =\frac{c-1-\phi_c(x)-\lambda}{c-\phi_c(x)-\lambda}
 =1-\frac1{c-\phi_c(x)-\lambda},\quad K\equiv(1-\partial_x^2)^{-1}.
$$
For $\lambda\le0$, $c-\phi_c(x)-\lambda>0$,
 and it is straightforward to show that \eqref{eq:FW-eigen} is equivalent to
\begin{equation}\label{eq:FW-p-ode}
 p_{xx}-A(x,\lambda)p=0,
 \qquad p\in H^2(\mathbb R).
\end{equation}
The coefficient $A$ is even in $x$, and
$$
 A(x,\lambda)\longrightarrow A_\infty(\lambda)
 :=1-\frac1{c-\lambda}>0
 \quad\text{as }x\to\pm\infty,\quad{\text{if }}\lambda\leq 0,
$$
since $1<c<4/3$.
Thus, for every $\lambda\in[\lambda_0,0]$, there is a unique
(up to a scalar factor) solution $p_-(x,\lambda)$ of
\eqref{eq:FW-p-ode} approaching 0 as $x\to-\infty$.
Normalize it by
\eq{
 p_-(x,\lambda)=e^{\kappa(\lambda)x}(1+o(1)),\qquad
 (p_-)_x(x,\lambda)=\kappa(\lambda)e^{\kappa(\lambda)x}(1+o(1)),
 \quad x\to-\infty,
}{pmas}
where $\kappa(\lambda)=\sqrt{A_\infty(\lambda)}>0$.

Introduce $\theta(x,\lambda)$ by
$$
 p_-=r\cos(\theta),\qquad (p_-)_x=r\sin(\theta),
 \qquad r=\sqrt{p_-^2+(p_-)_x^2}>0.
$$
Since both $p_-$ and $(p_-)_x$ are chosen to be positive for $x$ close to $-\infty$ in our normalization \eqref{pmas}, we choose, without loss of generality, the angle to satisfy
\eq{\theta(x,\lambda)\to\arctan(\kappa(\lambda))\in(0,\pi/2)
\text{ as }x\to-\infty.}{thas}

\emph{Eigenvalue criterion.}
Since $A(\cdot,\lambda)$ is even, a left-decaying solution
extends by reflection to an $L^2$ solution on the full line
if and only if it satisfies either the Neumann condition
$(p_-)_x(0,\lambda)=0$ (even extension) or the Dirichlet
condition $p_-(0,\lambda)=0$ (odd extension).
Conversely, any $L^2$ eigenfunction has one of these parities:
the space of solutions decaying at $-\infty$ is one-dimensional,
and {{the space of full-line decaying solutions is invariant under reflection}}.
Hence
\begin{equation}\label{eq:FW-angle-criterion}
 \lambda\in\sigma_p(L)\cap[\lambda_0,0]
 \quad\Longleftrightarrow\quad
 \theta(0,\lambda)\in\frac{\pi}{2}\mathbb Z.
\end{equation}
The values $\theta\in\pi\mathbb Z$ correspond to even modes,
while $\theta\in\pi/2+\pi\mathbb Z$ correspond to odd modes.

\emph{Strict monotonicity in the spectral parameter.}
Define the Wronskian
\[
 W(x,\lambda)
 :=p_-\,\partial_\lambda(p_-)_x-(p_-)_x\,\partial_\lambda p_-.
\]
Differentiating \eqref{eq:FW-p-ode} with respect to $\lambda$
gives the exact identity
\[
 \partial_xW(x,\lambda)
 =A_\lambda(x,\lambda)p_-^2(x,\lambda),
 \qquad
 A_\lambda(x,\lambda)
 =-\frac1{(c-\phi_c(x)-\lambda)^2}<0.
\]
The asymptotics for $p_-$ and $(p_-)_x$ and
differentiating the asymptotics of $p_-(x,\lambda)$ and $p_-'(x,\lambda)$ with respect to $\lambda$ shows that $W(x,\lambda)\to0$ as $x\to-\infty$.
Consequently,
\[
 W(x,\lambda)
 =-\int_{-\infty}^x
 \frac{p_-^2(s,\lambda)}{(c-\phi_c(s)-\lambda)^2}\,ds<0.
\]
On the other hand, differentiating $\theta(x,\lambda)$ with respect to $\lambda$ yields
\[
 \partial_\lambda\theta(x,\lambda)
 =\frac{p_-\,\partial_\lambda(p_-)_x
 -(p_-)_x\,\partial_\lambda p_-}
 {p_-^2+(p_-)_x^2}
 =-\frac{\displaystyle\int_{-\infty}^x
 \frac{p_-^2(s,\lambda)}{(c-\phi_c(s)-\lambda)^2}\,ds}
 {p_-^2+(p_-)_x^2}<0.
\]
Thus
\eq{
 {\partial_\lambda\theta(0,\lambda)<0
 \quad\text{for all }\lambda\in[\lambda_0,0].}
}{sm}

At $\lambda=\lambda_0$, \[ A(x,\lambda_0) =\frac{\phi_{\max}-\phi_c(x)} {1+\phi_{\max}-\phi_c(x)} \geq 0. \]
Since $p_-$ and $p_-'$ are positive near $-\infty$ (see \eqref{pmas}), while
\[ p_-''=A(x,\lambda_0)p_-, \qquad A(x,\lambda_0)\geq 0, \]
it follows that $p_->0$ and $p_-'>0$ on $(-\infty,0]$.
Hence, given the asymptotic state \eqref{thas} of $\theta$ as $x\to-\infty$
\[ 0<\theta(0,\lambda_0)<\frac{\pi}{2}. \]


At $\lambda=0$, differentiating $E'(\phi_c)+cQ'(\phi_c)=0$ with respect to $x$ gives $L\phi_c'=0$. Hence \[ p_0:=K\phi_c' \] solves \eqref{eq:FW-p-ode} at $\lambda=0$.
%
As $\phi_c$ is even, $p_0$ is odd. Furthermore, using the
positive, even convolution kernel $G(x)=e^{-|x|}/2$ and the
strict inequality $\phi_c'(s)<0$ for $s>0$, we have, for $x>0$,
\[
 \begin{aligned}
 p_0(x)
 &=\int_0^\infty
 \bigl[G(x-s)-G(x+s)\bigr]\phi_c'(s)\,ds<0,
 \end{aligned}
\]
where the inequality is a consequence of the fact that $G(x-s)>G(x+s)$ for $x,s>0$.
By oddness, $p_0(x)>0$ for $x<0$, and $p_0$ has exactly one
zero, at $x=0$. Thus $p_0>0$ on $(-\infty,0)$, $p_0(0)=0$, and, by uniqueness, $p_0'(0)<0$. Consequently, with the normalization \eqref{thas}, $\theta(0,0)=-\frac{\pi}{2}$,  not another
representative modulo $\pi$.
We conclude that
\begin{equation}\label{eq:FW-endpoint-angles}
 0<\theta(0,\lambda_0)<\frac\pi2,
 \qquad \theta(0,0)=-\frac\pi2.
\end{equation}


By continuity, the strict monotonicity \eqref{sm}, and \eqref{eq:FW-endpoint-angles}, there exists exactly one
$\lambda_*\in(\lambda_0,0)$ such that $ \theta(0,\lambda_*)=0.$ By \eqref{eq:FW-angle-criterion}, this
gives a negative eigenvalue. The only other value of $\frac{\pi}{2}\mathbb{Z}$
attained on $[\lambda_0,0]$ is $-\pi/2$, at $\lambda=0$, corresponding to the
translation mode. Since the space of solutions decaying at $-\infty$ is one-dimensional,
both eigenvalues are simple. Hence $ \ker L=\operatorname{span}\{\phi_c'\}$ and $n(L)=1$. 

}

\end{proof}

{{
\subsection{Note on the spectral properties}
\label{note}

As mentioned in subsection \ref{HF1}, the energy and the charge functionals defined in \eqref{Energy} and \eqref{charge} are
well defined on $\HT$ and twice Fr\'echet differentiable, with first and second
variations given in \eqref{crit} and \eqref{sec}. Thus one might consider
applying the Grillakis-Shatah-Strauss framework developed in \cite{GSS1987}
(see also Chapter~5 of \cite{KP}) to show orbital stability on $\HT$.

We here briefly recall some aspects of the Grillakis-Shatah-Strauss theory and show why
it cannot be applied here. Thus, this justifies the approach we need to take in this article and prove stability
in Theorem \ref{thm:local-L2-stability} with a mix of the $L^2$ and $L^\infty$ norms.

Let a Hamiltonian system be posed on a real Hilbert space $\mathcal{X}$,
with Hamiltonian $E$, and assume it admits a one-parameter group of
Hamiltonian symmetries $T(s):\mathcal{X}\to\mathcal{X}$,
$s\in\mathbb{R}$, with infinitesimal generator $T'(0)$.
We assume $T(s)$ acts unitarily on $\mathcal{X}$.
The evolution equation is
\begin{equation*}
u_t = \mathcal{J}\,\frac{\delta E}{\delta u},
\end{equation*}
where $\mathcal{J}:\mathcal{X}^*\to\mathcal{X}$ is a skew-symmetric Hamiltonian operator.
Solutions are understood in a weak sense: $u:\mathbb{R}\to\mathcal{X}$ satisfies
\begin{equation*}
{\frac{d}{dt}\langle \psi,u\rangle}
= -\left\langle \frac{\delta E}{\delta u},\mathcal{J}\psi\right\rangle,
\qquad \forall \psi\in D(\mathcal{J})\subset\mathcal{X}^*,
\end{equation*}
where $\langle\cdot,\cdot\rangle$ denotes the $\mathcal{X}$--$\mathcal{X}^*$
duality pairing.
The natural isomorphism $\Lambda:\mathcal{X}\to\mathcal{X}^*$ is defined by
\eq{
\langle \Lambda u,v\rangle=(u,v),
}{LD}
 with $(\cdot,\cdot)$ the inner product on $\mathcal{X}$.

%

Among other assumptions made in \cite{GSS1987}, one concerns spectral properties and can be written as follows:

\noindent\emph{Spectral structure.}
For each $\omega\in(\omega_1,\omega_2)$, the second variation
$\frac{\delta^2 {\mathcal{H}}}{\partial u^2}(\phi_c)$ has exactly one simple negative eigenvalue,
kernel $\mathrm{span}\{T'(0)\phi_c\}$, and the remainder of the spectrum
is strictly positive and bounded away from zero.

The linearized operator
{\[
{\mathcal{L}}\equiv \frac{\delta^2 {\mathcal{H}}}{\delta u^2}(\phi_c)
=\frac{\delta^2 E}{\delta u^2}(\phi_c)
-\omega\frac{\delta^2 Q}{\delta u^2}(\phi_c)
\]}
is self-adjoint from $\mathcal{X}$ to $\mathcal{X}^*$.
Its spectrum consists of the set
\eq{\sigma({\mathcal{L}})=
\left\{
\lambda\in\mathbb{R}:
{\mathcal{L}}-\lambda \Lambda{\text{ fails to be invertible}}
\right\}.
}{specD}
%

%
%
%

In our case, we have, from equation  \eqref{HF} and from sections \ref{HF1} and \ref{sec:2.2}, that
\eqnn{
\mathcal{J}=\partial_x,\;\;\mathcal{X}=\HT,\;\;\Lambda=1-\partial_x^2,\;\;\omega=-c,
}
the Hamiltonian symmetry $T$ is defined in \eqref{SymD}, \text{ and }$\phi_c$\text{ is the solution specified by Lemma \ref{lem-trav}}.

We verified the spectral assumptions directly for the operator $\mathcal{L}$.
However, according to the definition \eqref{specD}, the appropriate object is
the generalized eigenvalue problem
\[
\mathcal{L}v=\lambda \Lambda v,
\]
where, in the case of $\HT$, the isomorphism $\Lambda:\HT\to\HT^*$ satisfying
\eqref{LD} is given by $\Lambda=1-\partial_x^2$.
This corresponds to the operator denoted by $I$ in \cite[Eq.~(3.1) and paragraph after (2.17)]{GSS1987}.
Strictly speaking, one should therefore consider the spectrum of
\[
\widetilde{\mathcal{L}}\equiv \Lambda^{-1}\mathcal{L}:\HT\to\HT.
\]

In Lemma \ref{ess}, we found that the essential spectrum of $\mathcal L$ is positive and bounded away from zero. However, if one considers the operator $ \Lambda^{-1}\mathcal{L}$ together with the fact that the Fourier symbol of $ \Lambda^{-1}$ is $\frac{1}{1+r^2}$, it is straightforward, using the classical result of Henry \cite[Theorem A.2]{Henry81} as in the proof of Lemma \ref{ess}, to obtain that {the essential spectrum becomes $[0,\frac{c^2}{4}]$}. Thus the spectrum is not bounded away from the origin, and the formalism from  \cite{GSS1987} cannot be applied to prove the orbital stability on $\HT$. }}

%
%

{{\section{Slope condition}
\label{VK}

In the next section we prove orbital stability. In a similar fashion to how orbital stability is established when the Grillakis-Shatah-Strauss \cite{GSS1987} formalism applies, we need to obtain a slope condition for the conserved charge $Q$. This section is dedicated to the verification of that condition.

We consider the scalar function}}
{\eq{
d(c)\equiv E(\phi_c)+cQ(\phi_c)
}{ddef}}
with $\phi_c$ being the solution from Lemma \ref{lem-trav},
is convex, that is
$$
d''(c)>0
$$
{{on a certain interval of $c$-values.}}
We compute, using the fact that
$$
E'(\phi_c)+c\,Q'(\phi_c)=0,
$$
we find
\begin{equation}\label{eq5.1}
d'(c)=Q(\phi_c)=\frac{1}{2}\int_\mathbb{R}\phi_c^2\,dx.
\end{equation}

{In order to derive the explicit expression of $Q(\phi_c)$, we take advantage
of \eqref{eq:YTF-2.3}, which, for $k=0$, reads
$$
\frac{1}{2}(c-\phi_c)^2\left(\frac{d\phi_c}{dx}\right)^2=\frac{1}{8}\phi_c^2\left[\phi_c^2-\left(4c-\frac{8}{3}\right)\phi_c+4c(c-1)\right]=\frac{1}{8}\phi_c^2(\phi_c-\phi_-)(\phi_c-\phi_+),
$$
where
$$\phi_{\pm}=2c-\frac{4}{3}\pm\frac{2\sqrt{4-3c}}{3}$$
are the two positive roots of the quadratic polynomial
$$P(\phi)\equiv\phi_c^2-\left(4c-\frac{8}{3}\right)\phi_c+4c(c-1).$$ Note that for $x\in(-\infty,0)$, $\frac{d\phi_c}{dx}>0$, so we have
$$
\phi_c=\frac{2(c-\phi_c)}{\sqrt{(\phi_c-\phi_-)(\phi_c-\phi_+)}}\frac{d\phi_c}{dx},\ x\in(-\infty,0).
$$
Substituting this expression into \eqref{eq5.1} and taking advantage of the evenness of $\phi_c$, we have
\eq{
Q(\phi_c)&=\int_{-\infty}^{0}\phi_c^2\,dx \\
&=2\int_{0}^{\phi_-}\frac{\phi_c(c-\phi_c)}{\sqrt{(\phi_c-\phi_-)(\phi_c-\phi_+)}}d\phi_c. \\
}{dprime}
It should be noted that due to the presence of a singularity in the denominator, one cannot take derivative with respect to $c$ directly. Instead, setting
$$
z\equiv\sqrt{(\phi_c-\phi_-)(\phi_c-\phi_+)},\quad \alpha_{\pm}\equiv \frac{\phi_+\pm \phi_-}{2},\quad \gamma\equiv \phi_+\phi_-,
$$
and noting that
$$
d\phi_c=-\frac{zdz}{\alpha_+-\phi_c}, \quad \alpha_+-\phi_c=\sqrt{z^2+\alpha_-^2},\quad \alpha_+=2c-\frac{4}{3},\quad \alpha_-=\frac{2}{3}\sqrt{4-3c},\quad \gamma=4c(c-1).
$$
Note that $\alpha_+-\phi_c>0$ as a consequence of equation \eqref{phite} giving the maximum value of $\phi_c$.
Then \eqref{dprime} becomes
\eqnn{
Q(\phi_c)&=-2\int_{0}^{\sqrt{\gamma}}\frac{[(\alpha_+-\phi_c)-\alpha_+][(\alpha_+-\phi_c)-(\alpha_+-c)]}{\alpha_+-\phi_c}dz \\
&=-2\int_{0}^{\sqrt{\gamma}}\left\{(\alpha_+-\phi_c)-(2\alpha_+-c)+\frac{\alpha_+(\alpha_+-c)}{\alpha_+-\phi_c }\right\}dz\\
&=-2\int_{0}^{\sqrt{\gamma}}\left[\sqrt{z^2+\alpha_-^2}-(2\alpha_+-c)+\frac{\alpha_+(\alpha_+-c)}{\sqrt{z^2+\alpha_-^2}}\right]dz\\
&=-2\left[\frac{z\sqrt{z^2+\alpha_-^2}}{2}-(2\alpha_+-c)z+\left(\frac{\alpha_-^2}{2}+\alpha_+^2-c\alpha_+\right)\ln\left(z+\sqrt{z^2+\alpha_-^2}\right) \right]\Bigg|_{z=0}^{z=\sqrt{\gamma}}\\
&=8(c-1)\sqrt{c(c-1)}+\frac{4}{3}(c-1)(4-3c)\ln\left(\frac{3c-2+3\sqrt{c(c-1)}}{\sqrt{4-3c}}\right).
}
A straightforward calculation leads to
\eq{
d''(c)&=\frac{d}{dc}Q(\phi_c) \\
&=(16-\frac{4}{c})\sqrt{c(c-1)}+\frac{4}{3}(c-1)(4-3c)\left[\frac{c-\frac{1}{2}+\sqrt{c(c-1)}}{\sqrt{c(c-1)}(c-\frac{2}{3}+\sqrt{c(c-1)})}+\frac{3}{2(4-3c)}\right]\\
&  +\frac{4}{3}(7-6c)\ln\left(\frac{3c-2+3\sqrt{c(c-1)}}{\sqrt{4-3c}}\right)\\
&=14{\sqrt{c(c - 1)}}+\frac{4}{3}(7-6c)\ln\left(\frac{3c-2+3\sqrt{c(c-1)}}{\sqrt{4-3c}}\right).
}{eq5.4}
It is straightforward to check that the derivative of the argument of the logarithm above is increasing  for $c\in(1,\frac{4}{3})$ by calculating its derivative and it is equal to 1 at $c=1$. Thus the logarithmic expression above is positive and increasing. Hence for $1<c\leq7/6$, the expression above is positive. For $c\in (7/6,4/3)$, the second term in the expression above is negative and decreasing to $-\infty$. Thus there is an interval of values of $c$ in $(7/6,4/3)$  beyond $7/6$ not reaching $4/3$ for which the criterion above is still positive. This value is implicitly determined by the unique zero of the last expression of \eqref{eq5.4} on $(7/6,4/3)$ (which we denote by $c_0$) and the expression in \eqref{eq5.4} is positive for $c\in (1,c_0)$. Numerically, $c_0$ is found to be $c_0\approx 1.3333289$. {{We formalize our result in the following lemma.
\begin{lemma}\label{main}
Consider the value $c_0$ to be the unique zero of \eqref{eq5.4} on $(7/6,\,4/3)$ found numerically to be $c_0\approx 1.3333289$. Let $\phi_c$ be the smooth solitary wave solution of the FW equation \eqref{eq:FW}  as constructed in Lemma \ref{lem-trav}. Then the slope condition $d''(c)>0$ is satistifed for all $c\in(1,\,c_0)$, where the scalar function $d(c)$ is defined in \eqref{ddef}.
\end{lemma}}}

%

{{
\section{Orbital Stability}
\label{sec:L2-stability}

In this section, we obtain the orbital stability of the traveling wave solution given by the lemma \ref{lem-trav} for all $c\in(1,\,c_0)$. To do so,  we use  the results about the spectrum of the Hessian given in lemmas \ref{ess} and \ref{point}, and the slope condition given in Lemma \ref{main}.

The $H^1$ realization of the Hessian does not have an essential spectral gap. We therefore formulate the stability argument in $L^2(\mathbb R)$, where
\[
\mathcal L=(c-\phi_c)-(1-\partial_x^2)^{-1}
\]
has the spectral properties established in Lemmas~\ref{ess} and~\ref{point}. The nonlinear argument follows the strategy of \cite{Li2023}.

\subsection{Two estimate results}
\label{subsec:constrained-coercivity}

We now establish for our case a constrained coercivity estimate on the second variation of the Hessian $\mathcal{L}$ given in \eqref{sec2} that is used in other works in the context of orbital stability (see for example \cite[Proposition 3.2]{Li2023} and \cite[Equation 5.2.23]{KP}).

\begin{lemma}
\label{lem:L2-coercivity}
Fix $c\in(1,c_0)$, so that $d''(c)>0$. There exist constants
$\alpha>0$, $\beta>0$, and $\rho>0$ such that, whenever
$h\in L^2(\mathbb R)$ satisfies
\[
(h,\phi_c')_{L^2}=0,
\qquad
Q(\phi_c+h)=Q(\phi_c),
\qquad
\|h\|_{L^2}<\rho,
\]
one has
\begin{equation*}
\langle \mathcal Lh,h\rangle
\geq
\alpha\|h\|_{L^2}^2
-\beta\|h\|_{L^2}^3.
\end{equation*}
Moreover, the constants may be chosen uniformly when $c$ ranges over a
compact subinterval of $(1,c_0)$.
\end{lemma}

\begin{proof}
Differentiating
\[
E'(\phi_c)+cQ'(\phi_c)=0
\]
with respect to $c$, and using $Q'(\phi_c)=\phi_c$, gives
\begin{equation}
\label{eq:L-dc-phi}
\mathcal L\partial_c\phi_c=-\phi_c.
\end{equation}
Notice that the right-hand side of \eqref{eq:L-dc-phi} is the vector
$-Q'(\phi_c)=-\phi_c$, whereas
\[
\frac{d}{dc}Q(\phi_c)
=(\partial_c\phi_c,\phi_c)_{L^2}
=d''(c)
\]
is a scalar.

Since $\ker(\mathcal L)=\operatorname{span}\{\phi_c'\}$ and
$(\phi_c,\phi_c')_{L^2}=0$, the inverse of $\mathcal L$ applied to
$\phi_c$ is well defined on $\ker(\mathcal L)^\perp$. From
\eqref{eq:L-dc-phi},
\[
\mathcal L^{-1}\phi_c=-\partial_c\phi_c,
\]
and hence
\begin{equation*}
\langle \mathcal L^{-1}\phi_c,\phi_c\rangle
=-\langle \partial_c\phi_c,\phi_c\rangle
=-\frac{d}{dc}Q(\phi_c)
=-d''(c)<0.
\end{equation*}
It thus follows from standard results (See, for example, \cite[Theorem~5.3.2]{KP}; see also \cite{VK}) that
 there exists $\kappa_c>0$ such that
\begin{equation}
\label{eq:linear-constrained-coercivity}
\langle \mathcal Lv,v\rangle
\geq
\kappa_c\|v\|_{L^2}^2
\quad\text{whenever}\quad
(v,\phi_c)_{L^2}=(v,\phi_c')_{L^2}=0.
\end{equation}

Implementing the nonlinear condition on $h$, we use the expression for $Q$ given in \eqref{charge} to find that
the equality $Q(\phi_c+h)=Q(\phi_c)$ implies
\begin{equation}
\label{non}
(h,\phi_c)_{L^2}
=-\frac12\|h\|_{L^2}^2.
\end{equation}
Set
\eq{
q_c:=(\partial_c\phi_c,\phi_c)_{L^2}=d''(c)>0,
\qquad
a:=\frac{(h,\phi_c)_{L^2}}{q_c},
\qquad
v:=h-a\partial_c\phi_c.
}{set}
We claim that $v$ was chosen so that $(v,\phi_c')_{L^2}=(v,\phi_c)_{L^2}=0$. This claim is proven by a straightforward computation using the facts that we assume that {$(h,\phi_c')_{L^2}=0$ and we know that $(\partial_c\phi_c,\phi_c')_{L^2}=0$ }because $\partial_c\phi_c$ is even, whereas $\phi_c'$ is odd. Thus, $v$ was chosen to satisfy both orthogonality conditions in
\eqref{eq:linear-constrained-coercivity}.

We now compare the quadratic forms evaluated at $h$ and at $v$. Since
\[
h=v+a\partial_c\phi_c,
\]
we find
\begin{equation}
\label{eq:qform-decomposition}
\langle \mathcal Lh,h\rangle
=\langle \mathcal Lv,v\rangle-q_ca^2
\geq
\kappa_c\|v\|_{L^2}^2-q_ca^2.
\end{equation}
To obtain \eqref{eq:qform-decomposition} we used the fact that $\langle \mathcal Lv,\partial_c\phi_c\rangle
=\langle v,\mathcal L\partial_c\phi_c\rangle
=-(v,\phi_c)_{L^2}
=0$ and that $\langle \mathcal L\partial_c\phi_c,
        \partial_c\phi_c\rangle
=-(\phi_c,\partial_c\phi_c)_{L^2}
=-q_c$.


 From the triangular inequality applied to $v$ in \eqref{set}
\[
\|v\|_{L^2}
\geq
\|h\|_{L^2}-|a|\,\|\partial_c\phi_c\|_{L^2}.
\]
Squaring this estimate and discarding the resulting nonnegative last term,
we obtain
\eq{
\|v\|_{L^2}^2
&\geq
\|h\|_{L^2}^2
-2|a|\,\|h\|_{L^2}\|\partial_c\phi_c\|_{L^2},
}{vs}
Furthermore, from \eqref{non} and \eqref{set}, we have that $a=-\|h\|_{L^2}^2/(2q_c)$. Substituting this and \eqref{vs} in the inequality \eqref{eq:qform-decomposition}, one finds

%
%
\eqnn{
\langle \mathcal Lh,h\rangle
&\geq
\alpha\|h\|_{L^2}^2
-\beta_1\|h\|_{L^2}^3
-\beta_2\|h\|_{L^2}^4\\&\geq \alpha\|h\|_{L^2}^2
-\beta\|h\|_{L^2}^3,\;\;\;\text {if }\;\;\;\|h\|_{L^2}<1.
}
where
\eqnn{
\alpha=\kappa_c,\;\;\beta_1=\frac{\kappa_c}{d''(c)} \|\partial_c\phi_c\|_{L^2},\;\;\beta_2=\frac{1}{4\,d''(c)},\;\;\beta=\beta_1+\beta_2.
}

\end{proof}

The second estimate we prove is an $L^\infty$--$L^2$ estimate.
\begin{lemma}
\label{lem:Linfty-L2}
There exists a $C>0$ such that every
$h\in L^2(\mathbb R)\cap W^{1,\infty}(\mathbb R)$ satisfies
\begin{equation*}
\|h\|_{L^\infty}
\leq C\|h_x\|_{L^\infty}^{1/3}\|h\|_{L^2}^{2/3}.
\end{equation*}
\end{lemma}

\begin{proof}
Set $A=\|h\|_{L^\infty}$ and $K=\|h_x\|_{L^\infty}$. Choose $x_0\in\mathbb R$ such that $|h(x_0)|\geq A/2$. For $|x-x_0|\leq A/(4K)$,
\eqnn{
|h(x)|\geq |h(x_0)|-K|x-x_0|\geq \frac{A}{4}.
}
Thus,
\[
\|h\|_{L^2}^2
\geq \frac{A^2}{16}\frac{A}{2K}
=\frac{A^3}{32K},
\]
which proves the lemma.
\end{proof}

\subsection{Local orbital stability}

\begin{lemma}[Modulation near the solitary-wave orbit]
\label{lem:modulation}
Let $\phi_c\in H^1(\mathbb R)$ satisfy $\phi_c'\not\equiv0$. There exist
constants $\delta_0>0$ and $C>0$ such that, whenever $u\in L^2(\mathbb R)$
satisfies
\[
\inf_{y\in\mathbb R}
\|u-\phi_c(\cdot-y)\|_{L^2}
<\delta_0,
\]
there exists a translation parameter $r=r(u)\in\mathbb R$ for which
\[
h(x):=u(x+r)-\phi_c(x)
\]
satisfies
\begin{equation}
\label{eq:modulation-orthogonality}
(h,\phi_c')_{L^2}=0.
\end{equation}
Moreover,
\begin{equation}
\label{eq:modulation-estimate}
\|h\|_{L^2}
\leq
C\inf_{y\in\mathbb R}
\|u-\phi_c(\cdot-y)\|_{L^2}.
\end{equation}

\end{lemma}

\begin{proof}
{Choose $y_0\in\mathbb R$ such that
\begin{equation}
\label{phicn}
\|u(\cdot+y_0)-\phi_c\|_{L^2}
\leq
2\inf_{y\in\mathbb R}
\|u-\phi_c(\cdot-y)\|_{L^2}.
\end{equation}
Define
\[
G(w,a):=
\bigl(w(\cdot+a)-\phi_c,\phi_c'\bigr)_{L^2}.
\]
At $(w,a)=(\phi_c,0)$,
\[
G(\phi_c,0)=0,
\]
and differentiation with respect to $a$ gives
\[
\partial_a G(\phi_c,0)
=
(\phi_c',\phi_c')_{L^2}
=
\|\phi_c'\|_{L^2}^2>0.
\]
The implicit function theorem therefore provides a neighborhood of
$\phi_c$ in $L^2$ on which there is a unique
differentiable function $a=a(w)$ such that
\[
G(w,a(w))=0,
\qquad
|a(w)|\leq C\|w-\phi_c\|_{L^2}.
\]

Now, choose $w(x)=u(x+y_0)$ as close as necessary to $\phi_c$ by \eqref{phicn} by choosing $\delta_0$ small enough, and define
\[
r(u)=y_0+a(w),
\]
then
\[
h(x)=
u(x+r(u))-\phi_c(x)=
w(x+a(w))-\phi_c(x).
\]

Consequently, we have
\[
(h,\phi_c')_{L^2}=G(w,a(w))=0,
\]
which proves \eqref{eq:modulation-orthogonality}.

By the triangle
inequality and translation invariance of the $L^2$ norm,
\begin{align*}
&\|h\|_{L^2}
=
\|w(\cdot+a(w))-\phi_c\|_{L^2} \\
&\leq
\|w(\cdot+a(w))-\phi_c(\cdot+a(w))\|_{L^2}
+
\|\phi_c(\cdot+a(w))-\phi_c\|_{L^2} \\
&=
\|w-\phi_c\|_{L^2}
+
\|\phi_c(\cdot+a(w))-\phi_c\|_{L^2}.
\end{align*}
Since $\phi_c\in H^1(\mathbb R)$,
\[
\|\phi_c(\cdot+a)-\phi_c\|_{L^2}
\leq
|a|\,\|\phi_c'\|_{L^2}.
\]
Using the implicit-function estimate for $a(w)$ gives
\[
\|h\|_{L^2}
\leq
C\|w-\phi_c\|_{L^2},
\]
with $w(x)=u(x+y_0)$ for any $y_0$ satisfying \eqref{phicn}.
Therefore
\[
\|h\|_{L^2}
\leq
C\inf_{y\in\mathbb R}
\|u-\phi_c(\cdot-y)\|_{L^2} ,
\]
which proves \eqref{eq:modulation-estimate}. }
\end{proof}

Before we state the orbital stability theorem, we mention the well-posedness result of \cite[Theorem 2.2]{Yin2005} discussed in the Introduction, that for initial data $u_0\in H^s(\mathbb{R})$, $s>3/2$,
there exists $T>0$ such that $u\in C([0,T];H^s(\mathbb{R}))$.

\begin{theorem}
\label{thm:local-L2-stability}
Fix $c\in(1,c_0)$ and $s>3/2$. Let
$\phi_c$ be the solitary wave of Lemma~\ref{lem-trav}.
Consider $u_0\in H^s(\mathbb R)$. {{Let $T>0$ be  such that}}
$u\in C([0,T];H^s(\mathbb R))$ is a solution of the
Fornberg--Whitham equation with $u(0)=u_0$. Then, for
every sufficiently small $\varepsilon>0$, there exists $\delta>0$ with the
following property. If
\eq{
\|u_0-\phi_c\|_{L^2}<\delta,
}{epde}
then
\begin{equation}
\label{eq:local-L2-stability}
\sup_{0\leq t\leq T}
\inf_{r\in\mathbb R}
\|u(t)-\phi_c(\cdot-r)\|_{L^2}
<\varepsilon.
\end{equation}
Moreover, there exists a constant ${C}>0$ such that
\begin{equation}
\label{eq:local-Linfty-stability}
\sup_{0\leq t\leq T}
\inf_{r\in\mathbb R}
\|u(t)-\phi_c(\cdot-r)\|_{L^\infty}
\leq {C}\varepsilon^{2/3}.
\end{equation}
Here $\delta$ depends {{on $c$, $\varepsilon$, and $T$.}}
\end{theorem}

Note: the solution is taken to be in $H^s$, $s>3/2$ to take advantage of the local well-posedness discussed {{above}}, and also of the fact that the  functionals defined in \eqref{Energy} and \eqref{charge} are twice Fr\'echet differentiable on $\HT$.

\begin{proof}
We divide the proof into four steps.

\medskip
\noindent
\textbf{Step 1: Matching the charge with a nearby solitary wave.}
Since
\[
\frac{d}{dc}Q(\phi_c)=d''(c)>0,
\]
the map $c\mapsto Q(\phi_c)$ is locally invertible. Thus, after decreasing
$\delta$ if necessary, there exists a unique parameter $\widehat c$ near
$c$ such that
\begin{equation}
\label{eq:matched-charge}
Q(u_0)=Q(\phi_{\widehat c}).
\end{equation}
It then follows from the facts that  $c\mapsto Q(\phi_c)$ is smooth, the functional $Q$ is smooth on $L^2(\mathbb{R})$, and the map $ c \mapsto \phi_c$ is smooth into $L^2(\mathbb{R})\cap L^\infty(\mathbb{R})$, that
\eq{\left|c-\widehat{c}\,\right|\leq C\|u_0-\phi_c\|_{L^2}\;\;{\text and }\;\;\|\phi_{\widehat c}-\phi_c\|_{L^\infty} +\|\phi_{\widehat c}-\phi_c\|_{L^2}
\leq C\|u_0-\phi_c\|_{L^2}, }{eq:profile-c-hat-estimate}
for some $C>0$. The first of the inequalities in \eqref{eq:profile-c-hat-estimate} implies that, because of \eqref{epde}, in this proof, we can restrict the value of $\widehat{c}$  to any closed interval satisfying
\eq{
\widehat{c}\in I\equiv\left[ c-\kappa,\,c+\kappa\right]\subset (1,c_0)\;\;\text{ such that }\;\;\kappa>C\delta.
}{Idef}

\medskip
\noindent
\textbf{Step 2: Modulation and the nonlinear constraints.}
For as long as $u(t)$ remains in the neighborhood of the orbit on which
Lemma~\ref{lem:modulation} applies, choose $r(t)$ so that
\[
h(t,x):=u(t,x+r(t))-\phi_{\widehat c}(x)
\]
satisfying conditions \eqref{eq:modulation-orthogonality} and \eqref{eq:modulation-estimate}.
At $t=0$, choosing $r=0$ in the orbital distance and using
\eqref{eq:profile-c-hat-estimate} gives
\begin{equation}
\label{eq:h0-bound}
\|h(0)\|_{L^2}
=
\|u_0-\phi_{\widehat c}\|_{L^2}
\leq
C\|u_0-\phi_c\|_{L^2}.
\end{equation}

Since $Q$ is translation invariant and conserved,
\eqref{eq:matched-charge} yields
\begin{align*}
Q(\phi_{\widehat c}+h(t))
=Q(u(t,\cdot+r(t)))
=Q(u(t))=Q(u_0)=Q(\phi_{\widehat c}).
\end{align*}
Thus $h$ satisfies the comditions in
Lemma~\ref{lem:L2-coercivity}, for $c=\widehat c$.

\medskip
\noindent
\textbf{Step 3: Lower and upper bounds for the conserved action.}
Consider the Lyapunov functional \eqref{L} for the specific value of the velocity $c=\widehat{c}$
\[
\mathcal H_{\widehat c}(u):=E(u)+\widehat c \,Q(u).
\]
From the definition of the energy and the charge given in \eqref{Energy} and \eqref{charge}, we have
\begin{equation}
\label{eq:exact-action-expansion}
\mathcal H_{\widehat c}(\phi_{\widehat c}+h)
-
\mathcal H_{\widehat c}(\phi_{\widehat c})
=
\frac12\langle \mathcal L_{\widehat c}h,h\rangle
-
\frac16\int_{\mathbb R}h^3\,dx.
\end{equation}
By Lemma~\ref{lem:L2-coercivity},
\begin{equation*}
\frac12\langle \mathcal L_{\widehat c}h,h\rangle
\geq
\frac{\alpha}{2}\|h\|_{L^2}^2
-
\frac{\beta}{2}\|h\|_{L^2}^3.
\end{equation*}
Moreover,
\begin{equation}
\label{eq:cubic-holder}
\left|\int_{\mathbb R}h^3\,dx\right|
\leq
\|h\|_{L^\infty}\|h\|_{L^2}^2.
\end{equation}

Then by using Lemma \ref{lem:Linfty-L2} and the fact that the $L^\infty$ norm of $h_x$ is uniformly bounded on the closed $c$ interval
\footnote{{Due to the fact that $u\in C([0,T];H^s(\mathbb{R}))$, and hence, by Sobolev embedding, $u_x\in C([0,T];L^\infty(\mathbb{R}))$ and thus $\|u_x\|_\infty$ is uniformally bounded on compact intervals.}}
 \eqref{Idef},
\begin{equation}
\label{eq:lower-action}
\mathcal H_{\widehat c}(\phi_{\widehat c}+h)
-
\mathcal H_{\widehat c}(\phi_{\widehat c})
\geq
\frac{\alpha}{2}\|h\|_{L^2}^2
-
\frac{\beta}{2}\|h\|_{L^2}^3
-
{C}\|h\|_{L^2}^{8/3},
\end{equation}
for some ${C}>0$.
Factoring $\|h\|_{L^2}^2$ from the right-hand side gives
\begin{align*}
\frac{\alpha}{2}\|h\|_{L^2}^2
-
\frac{\beta}{2}\|h\|_{L^2}^3
-
{C}\|h\|_{L^2}^{8/3} 
=
\left(
\frac{\alpha}{2}
-
\frac{\beta}{2}\|h\|_{L^2}
-
{C}\|h\|_{L^2}^{2/3}
\right)
\|h\|_{L^2}^2.
\end{align*}
Choose $0<\rho_1\leq\rho$ sufficiently small that
\begin{equation*}
\frac{\beta}{2}\rho_1+{C}\rho_1^{2/3}
\leq
\frac{\alpha}{4}.
\end{equation*}
{Then, whenever $\|h\|_{L^2}\leq\rho_1$, the coefficient in parentheses in
\eqref{eq:lower-action} is at least $\alpha/4$.} Hence
\begin{equation}
\label{eq:strict-lower-action}
\mathcal H_{\widehat c}(\phi_{\widehat c}+h)
-
\mathcal H_{\widehat c}(\phi_{\widehat c})
\geq
\frac{\alpha}{4}\|h\|_{L^2}^2.
\end{equation}

We next derive the initial upper bound. At $t=0$, by Lemma \ref{Lemspec} and the self-adjointness of $\mathcal L_{\widehat c}$ on $L^2$,
\[
\frac12
\left|
\langle \mathcal L_{\widehat c}h(0),h(0)\rangle
\right|
\leq
C_l\|h(0)\|_{L^2}^2,
\]
for some $C_l\geq 0$ for all $\widehat{c}$ on the closed interval $I$ in \eqref{Idef}.
Also, \eqref{eq:cubic-holder} and Lemma \ref{lem:Linfty-L2} give
\[
\frac16\left|\int_{\mathbb R}h(0)^3\,dx\right|
\leq
{C}\|h(0)\|_{L^2}^{8/3},
\]
where the ${C}$ is the same as the one in \ref{eq:lower-action}.
Using the exact expansion \eqref{eq:exact-action-expansion}, and enlarging
${C}$ if necessary so that ${C}\geq C_l$, we obtain
\begin{equation}
\label{eq:initial-upper-action}
\left|
\mathcal H_{\widehat c}(u_0)
-
\mathcal H_{\widehat c}(\phi_{\widehat c})
\right|
\leq
{C}\left(
\|h(0)\|_{L^2}^2
+
\|h(0)\|_{L^2}^{8/3}
\right).
\end{equation}
By \eqref{eq:h0-bound}, the RHS tends to zero as
$\|u_0-\phi_c\|_{L^2}\to 0$.

\medskip
\noindent
\textbf{Step 4: Continuation argument and return to the original wave.}
By the smoothness of the conserved functionals and by \eqref{eq:h0-bound} and \eqref{eq:initial-upper-action}, we can choose $\delta>0$
so small that
\begin{equation}
\label{eq:initial-action-small}
\left|
\mathcal H_{\widehat c}(u_0)
-
\mathcal H_{\widehat c}(\phi_{\widehat c})
\right|
<
\frac{\alpha}{16}
\min\{\varepsilon^2,\rho_1^2,\delta_0^2\},\;\;\text{ and }\;\;\|h(0)\|_{L^2}<
\eta\equiv\min\left\{
\frac{\varepsilon}{2},
\frac{\rho_1}{2},
\frac{\delta_0}{2}
\right\},
\end{equation}
where $\delta_0$ is included so that Lemma \ref{lem:modulation} applies and $\rho_1$ is for the bound  \eqref{eq:strict-lower-action} to hold. We claim that $\|h(t)\|_{L^2}<\eta$.

Suppose, by contradiction, that $\|h(t)\|_{L^2}$ reaches for a first time either
$\varepsilon/2$, $\rho_1/2$, or $\delta_0/2$ at $t=t_*$. Thus \eqref{eq:strict-lower-action} applies
at $t=t_*$ since $\|h(t_*)\|_{L^2}\leq \rho_1$. Conservation of the action and
\eqref{eq:initial-action-small} give
{\begin{align*}
\frac{\alpha}{4}\|h(t_*)\|_{L^2}^2
&\leq
\mathcal H_{\widehat c}(\phi_{\widehat c}+h(t_*))
-
\mathcal H_{\widehat c}(\phi_{\widehat c}) \\
&=
\mathcal H_{\widehat c}(u_0)
-
\mathcal H_{\widehat c}(\phi_{\widehat c}) \\
&<
\frac{\alpha}{16}
\min\{\varepsilon^2,\rho_1^2,\delta_0^2\}.
\end{align*}}
It follows that
\[
\|h(t_*)\|_{L^2}
<
\frac12\min\{\varepsilon,\rho_1,\delta_0\},
\]
contradicting the definition of $t_*$. Therefore
\begin{equation}
\label{eq:stability-chat}
\sup_{0\leq t\leq T}
\inf_{r\in\mathbb R}
\|u(t)-\phi_{\widehat c}(\cdot-r)\|_{L^2}
<
\frac{\varepsilon}{2}.
\end{equation}

By \eqref{eq:profile-c-hat-estimate}, we may decrease $\delta$ further so
that
\[
\|\phi_{\widehat c}-\phi_c\|_{L^2}<\frac{\varepsilon}{2}.
\]
The triangle inequality and \eqref{eq:stability-chat} then yield
\eqref{eq:local-L2-stability}.

Finally, Lemma \ref{lem:Linfty-L2}, \eqref{eq:modulation-estimate}, and
\eqref{eq:stability-chat} imply
\[
\sup_{0\leq t\leq T}\|h(t)\|_{L^\infty}
\leq
{C}\varepsilon^{2/3},
\]
where ${C}$ is the same as in \eqref{eq:lower-action}, making that constant larger if necessary.
Using \eqref{eq:profile-c-hat-estimate} once more, and decreasing $\delta$
so that
\[
\|\phi_{\widehat c}-\phi_c\|_{L^\infty}
\leq
{C}\varepsilon^{2/3},
\]
proves \eqref{eq:local-Linfty-stability}, again after enlarging ${C}$ if
necessary.
\end{proof}

}}

\section{Conclusion}
\label{conc}

In this paper, we have proven Theorem \ref{main} concerning the orbital stability of the smooth solitary wave solution of Lemma \ref{lem-trav}.

{{Let us add here that we cannot conclude that the solutions that correspond to the small interval $c\in(c_0,4/3)$ will be unstable due to the criterion in \eqref{eq5.4} being negative.
It would thus be}} of interest to investigate to check for instability in this parameter regime in a future work. One natural approach would be to analyze the spectrum of the linearized FW equation about the solitary wave. This could be initiated by deriving energy--estimate--type bounds to confine any unstable eigenvalues to a bounded region of the complex plane (see, for example, \cite{Humpherys02} for related techniques). One could then compute the Evans function numerically \cite{Evans, Pego92, Sandstede} in order to track the movement of eigenvalues as the wave speed $c$ crosses the critical threshold $c=c_0$, and in particular to identify eigenvalues that transition from the left half--plane into the unstable right half--plane.

{As we mention in Section \ref{sec:2.2},  as $c$ goes to $4/3$,
  smooth solitary waves degenerates into waves of peakon form. Since the smooth soliton solution is unstable for $c_0<c<4/3$, we would expect the peakon to be unstable as well. It would thus be interesting to obtain the spectral stability properties of the peakon solution as it was done for example in the case of the $b$-family \cite{LafortunePelinovsky2022,CharalampidisParkerKevrekidisLafortune2023}}.

Due to the symmetry \eqref{sym}, this stability result applies to the nonzero background solutions $\phi_k$ constructed in Lemma \ref{lem-travi}. Given that our stability result Theorem \ref{main} applies only to $c\in (1,\,c_0)$, the result for $\phi_k$ applies to any value of $c$ with $k\in (1-c_0,1)$.

{{The reason we were not able to simply apply the Grillakis-Shatah-Strauss framework \cite{GSS1987} is that the space $\HT$ is discussed in subsection \ref{note} and is due to the spectral properties of the Hessian not meeting the correct assumptions. Due to the presence of a cubic term in the energy functional
$E$ defined in \eqref{Energy},}} we had to adapt a similar strategy as what was done for the
Degasperis--Procesi equation, whose energy also contains a cubic contribution.
In that setting,  the authors of \cite{Li2023} were able to work in
$L^2(\mathbb{R})$ by controlling the $L^\infty$-norm of the solution in terms of
its $L^2$-norm. This was achieved by imposing additional regularity on the
initial data (see \cite[Propositions~2.2 and~2.3]{Li2023}), following ideas
developed in \cite[Lemma~2]{MolinetDP}. A crucial ingredient in this approach is
the fact that the momentum variable $m=u-u_{xx}$ remains nonnegative for all
times for which the solution exists. {{In our case, we used the $L^\infty-L^2$ estimate given by Lemma \ref{lem:Linfty-L2}}}

For the whole $b$-family of equations \cite{Holm1,Holm2}, which includes the
Camassa--Holm equation \cite{ch,ch2} and the Degasperis--Procesi equation
\cite{dp,dhh}, this sign preservation property indeed holds. In terms of the
momentum variable
\[
m \equiv u - u_{xx},
\]
the $b$-family can be written as a closed transport--stretching equation
\eq{
m_t + u\,m_x + b\,u_x\,m = 0,
\qquad b=2 \text{ (CH)},\; b=3 \text{ (DP)}.
}{bf}
Let $X(t;x_0)$ denote the characteristic flow generated by the velocity field
$u$, defined by $\dot X(t)=u(X(t),t)$ with $X(0)=x_0$. Along such characteristics,
the momentum satisfies the scalar ordinary differential equation
\[
\frac{d}{dt} m(X(t),t)
= - b\,u_x(X(t),t)\, m(X(t),t),
\]
which admits the explicit solution
\[
m(X(t),t)
= m_0(x_0)\exp\!\left(-b\int_0^t u_x(X(s),s)\,ds\right).
\]
Since the exponential factor is strictly positive as long as the solution
remains smooth, the sign of $m$ is preserved along characteristics. In
particular, if the initial datum satisfies $m_0\ge0$, then $m(\cdot,t)\ge0$ for
as long as the solution exists classically. This invariance of the region
$\{m\ge0\}$ under the flow plays a fundamental role in the global existence and
stability theory for the Camassa--Holm and Degasperis--Procesi equations.

By contrast, for the Fornberg--Whitham equation the corresponding momentum
equation contains an additional forcing term,
\[
m_t + u\,m_x + 3u_x\,m
=
3u\,u_x - u_x,
\qquad m = u - u_{xx},
\]
which destroys the purely multiplicative structure of the evolution along
characteristics and prevents any analogous sign--preservation mechanism.

\section{Data Availability}

The data that support the findings of this study are available within the article.


\section*{Acknowledgements}

{{This research was partially supported by the NSFC (No.12571172)} and the
Scientific Research Fund of Hunan Provincial Education Department
(No.21A0414).
 The research of Z. Liu was supported by the NSFC (No.12571188), and Guangdong Basic and Applied Basic Research
Foundation (No.2024A1515012704).}} {{The work of S. Lafortune was partially supported by the National Science Foundation under Grant No. DMS-2606586.}}

\end{document}